\documentclass[11pt, reqno]{amsart}% \documentclass{article}[11pt]

\usepackage{amssymb,amsmath,amsfonts,amsthm,array,comment,mathrsfs,times,graphicx}
\usepackage{stackrel}
\usepackage[latin1]{inputenc}
\usepackage{enumitem}
\usepackage{float}
\usepackage{verbatim}
\usepackage{makeidx}
\makeindex
\usepackage{geometry}
\usepackage{longtable,tabularx}
\usepackage{amsmath}
\usepackage[all,cmtip]{xy}
\usepackage{tikz-cd}
\DeclareFontFamily{U}{wncy}{}
    \DeclareFontShape{U}{wncy}{m}{n}{<->wncyr10}{}
    \DeclareSymbolFont{mcy}{U}{wncy}{m}{n}
    \DeclareMathSymbol{\Sha}{\mathord}{mcy}{"58}

\newtheorem{theorem}{Theorem}[section]
\newtheorem{corollary}[theorem]{Corollary}
\newtheorem{lemma}[theorem]{Lemma}
\newtheorem{proposition}[theorem]{Proposition}

\newtheorem{hypothesis}[theorem]{Hypothesis}
\theoremstyle{definition}

\newtheorem{remark}[theorem]{Remark}

\newtheorem{examples}[theorem]{Examples}
\title{On the existence of Minkowski units}

\begin{document}

\author{David Burns, Donghyeok Lim and Christian Maire}

\begin{abstract} We investigate the Galois structure of algebraic units in cyclic extensions of number fields and thereby obtain strong new results on the existence of independent Minkowski $S$-units. \end{abstract} 

\address{King's College London,
Department of Mathematics,
London WC2R 2LS,
U.K.}
\email{david.burns@kcl.ac.uk}

\address{FEMTO-ST Institute, Universit\'{e} Franche-Comt\'{e}, CNRS,  15B avenue des Montboucons, 25000 Besan\c{c}on, FRANCE}
\email{donghyeokklim@gmail.com}

\address{FEMTO-ST Institute, Universit\'{e} Franche-Comt\'{e}, CNRS,  15B avenue des Montboucons, 25000 Besan\c{c}on, FRANCE}
\email{christian.maire@univ-fcomte.fr}
\maketitle

\section{Introduction}
We fix a finite Galois extension of number fields $L/K$, with $G := \mathrm{Gal}(L/K)$, and an odd prime divisor $p$ of the order of $G$ and write $\mathbb{F}_p$ for the finite field of cardinality $p$. For a finite set of places $S$ of $K$, we write $\mathcal{O}_{L,S}$ for the ring of algebraic $S$-integers of $L$ and, with $\mu_L$ denoting the group of roots of unity in $L$, define a $\mathbb{Z}_p[G]$-lattice by setting $U_{L,S} := \mathbb{Z}_p\otimes_\mathbb{Z} (\mathcal{O}_{L,S}^\times/\mu_L)$. 

If the $G$-module $U_{L,S}/U_{L,S}^p$ has a direct summand isomorphic to $\mathbb{F}_p[G]^{m}$ for a natural number $m$, then one says $L/K$ has a family of `$m$ independent Minkowski $S$-units'. In particular, by the Krull-Schmidt Theorem, the maximum size $m_{L/K,S}$ of a family of independent Minkowski $S$-units for $L/K$ is well-defined. Recent work in \cite{HMR1} has shown that $m_{L/K,S}$ plays an important role in the study of both tamely ramified pro-$p$ extensions and the deficiency of $p$-class tower groups and also, following work of Ozaki \cite{Ozaki}, of the inverse Galois problem for the $p$-class field tower (cf. \cite{HMR2}). Unfortunately, however, the determination of $m_{L/K,S}$ appears, in general, to be a very difficult problem.

%Whilst the explicit determination of $m_{L/K,S}$ is a very difficult problem, the same authors have also described, in \cite{HMR3}, several ad hoc constructions in which the presence of many places that ramify tamely in $L/K$ can force $m_{L/K,S}$ to be `large'. 

In this note we identify conditions under which one can `bound' the complexity of the $\mathbb{Z}_p[G]$-lattice $U_{L,S}$ and thereby deduce new results on $m_{L/K,S}$. Here we recall that understanding the explicit structure of arithmetic lattices is a notoriously difficult problem, not least because, by well-known results of Heller and Reiner \cite{hr, HellerReiner2}, the relevant theory of integral representations is usually extremely complicated. 

To recall the most general (as far as we are aware) result in this direction, we fix an abstract finite group $\Gamma$, a finite set of places $\Sigma$ of $K$ containing all $p$-adic places and a $p$-adic Galois representation $T$ over $K$ unramified outside $\Sigma$. Then \cite[Th. 1.1]{Burns} proves the existence of an upper bound on the number of isomorphism classes of indecomposable modules that occur in the Krull-Schmidt decompositions of the $\mathbb{Z}_p[\Gamma]$-lattices arising from the $p$-adic \'etale cohomology groups $H^i({\rm Spec}(\mathcal{O}_{L',\Sigma})_{{\rm \acute e
t}},T) \cong H^i({\rm Spec}(\mathcal{O}_{K',\Sigma})_{{\rm \acute e
t}},\mathbb{Z}_p[\Gamma]\otimes_{\mathbb{Z}_p}T)$ as $L'/K'$ ranges over extensions of $K$ for which $L'/K'$ is unramified outside $\Sigma$, Galois and such that $\mathrm{Gal}(L'/K')$ identifies with $\Gamma$. In particular, in the case ($i=1$ and $T = \mathbb{Z}_p(1))$  relevant to us, this result relates to the module $U_{L,S}$ only if $S$ contains all places that are $p$-adic or ramify in $L$ and gives bounds depending on $\Gamma$ and the number of places of $L$ that are $p$-adic or divide the different of $L/K$. 

In contrast, by combining detailed class-field theoretic arguments together with algebraic results of Yakovlev \cite{Yakovlev1}, we shall now obtain the finer information about $U_{L,S}$ that is contained in the following result. In this result, for each natural number $n$ we write $\mathcal{C}_n$ for the (countably infinite) collection of pairs $(L/K,S)$ comprising a Galois extension of number fields $L/K$ for which $G$ identifies with the cyclic group $\mathbb{Z}/p^n$ of order $p^n$ and ${\rm Norm}_{L/K}(\mu_L) = \mu_K$, and a finite set $S$ of places of $K$ for which the $S$-ideal $p$-class group of every intermediate field of $L/K$ is cyclic. %, we shall prove the following result. %To be specific, in Theorem \ref{finitude} we shall prove a precise version of the following result. 

\begin{theorem}\label{intro thm} Fix a natural number $n$. Then, as $(L/K,S)$ ranges over $\mathcal{C}_n$, only finitely many isomorphism classes of indecomposable $\mathbb{Z}_p[(\mathbb{Z}/p^n)]$-lattices arise as direct summands of any $U_{L,S}$. \end{theorem}

At the outset, we note that the conclusion of this result is, a priori, far from clear since the $\mathbb{Z}_p$-rank of $U_{L,S}$ is unbounded as $(L/K,S)$ ranges over $\mathcal{C}_n$, whilst, if $n > 2$, then there exist infinitely many non-isomorphic indecomposable $\mathbb{Z}_p[(\mathbb{Z}/p^n)]$-lattices (cf. \cite{HellerReiner2}). In addition, Theorem \ref{intro thm} is stronger than the corresponding case of \cite[Th. 1.1]{Burns} since, firstly, its conclusion does not require $S$ to contain all places that are either $p$-adic or ramify in $L$ and, secondly, its proof gives more information on the occurring indecomposable modules and thereby leads both to sharper bounds on the number of such isomorphism classes and also, upon appropriate specialisation, to some very concrete consequences. For example, if the $p$-Hilbert $S$-class field of $L$ is cyclic over $K$, then it implies the $\mathbb{Z}_p[G]$-structure of $U_{L,S}$ depends only on the ramification and residue degrees of places of $K$ that are ramified in $L$ or belong to $S$ and can even be described completely explicitly if $L/K$ is unramified  (see Theorem \ref{last thm}). 

These improvements also mean that Theorem \ref{intro thm} can be used to deduce the existence of families of extensions in which $m_{L/K,S}$ is unbounded even though the set of places ramifying in $L/K$ remains small and contains no place that is tamely ramified, thereby complementing the constructions of \cite{HMR3}. (For details see Corollary \ref{cft remark} and Examples \ref{last exams}).

We remark that several aspects of the techniques developed here can be extended to more general classes of extensions $L/K$ (thereby further refining the general approach of \cite{Burns}). Such results are discussed both in the article \cite{BouzzaouiLim} of Bouazzaoui and the second author and in forthcoming work \cite{KumonLim} of Kumon and the second author. %In 
% , and write $\mathcal{O}_L$ for the ring of algebraic integers of $L$. Then, for each prime divisor $p$ of $|G|$, obtaining information about the explicit structure of $U_L := \mathbb{Z}_p\otimes_\mathbb{Z}\mathcal{O}_L^\times$ as a $\mathbb{Z}_p[G]$-module is a classical, and very difficult, problem and, even now, there are no general results (cf. \cite{BouzzaouiLim, KumonLim} and the references contained therein).  

%Throughout, we shall use the following general notation. 

Finally, for the reader's convenience, we record some general notation. For a Galois extension of fields $F/E$, we abbreviate $\mathrm{Gal}(F/E)$ to $G(F/E)$. For a finitely generated $\mathbb{Z}_p$-module $M$ we write $\mathrm{rk}(M)$ for its rank $\mathrm{dim}_{\mathbb{Q}_p}(\mathbb{Q}_p \otimes_{\mathbb{Z}_p} M)$. For an abelian group $X$ we set $X_p := \mathbb{Z}_p\otimes_{\mathbb{Z}}X$. %To simplify the notation, we use $\mu_{E,p}$ instead of $(\mu_E)_p$ for a number field $E$. 
For a natural number $n$, we set $[n]:= \{i\in \mathbb{Z}:1\le i\le n\}$ and $[n]^\ast := \{0\}\cup [n]$.

\vskip 15pt
\textbf{Acknowledgement} We are very grateful to Ozaki Manabu and Ravi Ramakrishna for their interest in our work and helpful comments.

\section{Hypotheses and examples}
At the outset we fix an odd prime number $p$. For an extension of number fields $L/K$ and finite set $S$ of places of $K$, we write  $A_{L,S}$ for the Sylow $p$-subgroup of the $S$-ideal class group of $L$ (that is, the quotient of the ideal class group of $L$ by the subgroup generated by the classes of prime ideals lying above $S$), $H_{L,S}$ for the $p$-Hilbert $S$-class field of $L$ (that is, the maximal unramified abelian $p$-extension of $L$ in which all places of $L$ above $S$ split completely) so that $A_{L,S}$ is canonically isomorphic to $G(H_{L,S}/L)$, and $R_{L/K}$ for the set of places of $K$ that ramify in $L$. If $S=\emptyset$, then we abbreviate $H_{L,S}$, $A_{L,S}$ and $U_{L,S}$ to $H_L$, $A_L$ and $U_L$ respectively. We also write $K_S$ for the maximal pro-$p$ extension of $K$ unramified outside $S$ and set $G_{K,S} := \mathrm{G}(K_{S}/K)$. 

%\subsection{The hypotheses}

We now fix the following data: 
\begin{equation}\label{data fix}
\begin{cases}&\text{a finite cyclic $p$-extension of number fields $L/K$ with Galois group $G$};\\
             &\text{a finite set $S$ of places of $K$ with $S \cap R_{L/K} = \emptyset$.}\end{cases}\end{equation}
We assume that this data satisfies the following hypothesis. % investigate the $G$-structure of $\mathcal{O}_{L,S}^\times$ under the following hypothesis.

\begin{hypothesis}\label{key hyp}\ %The data $(L/K,S)$ has the following properties: % The following conditions are satisfied: 
\begin{itemize}
\item[(C1)] For every intermediate subfield $E$ of $L/K$, the group $A_{E,S}$ is cyclic.
\item[(C2)] ${\rm Norm}_{L/K}(\mu_{L}) = \mu_K$.% is $G_{L/F}$-cohomologically trivial (i.e. $\hat{H}^{i}(H, \mu_{L})=0$ for any subgroup $H$ of $G_{L/F}$ and $i \in \mathbb{Z}$).
\end{itemize}
\end{hypothesis}

\begin{remark}\label{herbrand rem} Condition (C2) has a conceptual interpretation: since $G$ is a cyclic $p$-group and $\mu_L$ is finite, a Herbrand quotient argument can be combined with general results (cf. \cite[Th. 5 and Th. 9]{aw}) to show (C2) is satisfied if and only if $\mu_L$ is a cohomologically-trivial $G$-module. In addition, since $p$ is odd, a straightforward analysis also shows that the latter condition is satisfied if and only if either the Sylow $p$-subgroup $\mu_{K,p}$ of $\mu_K$ is trivial or one has $L=K(\mu_{L,p})$ (see, for example, \cite[Lem. 5.4.4(1)]{Popescu}). 
\end{remark}

It is clear that, for any given $L/K$, the Chebotarev Density Theorem implies that one can simply increase the set $S$ in order to satisfy (C1). On the other hand, for several natural families of extensions $L/K$, such as in the following examples, (C1) is satisfied with $S=\emptyset$. 

\begin{examples} In each of the following cases, the extension $L/K$ is tamely ramified.  

\noindent{}(i) Assume $A_K$ is cyclic and non-trivial. By the Burnside Basis Theorem, $G_{K,\emptyset}$ is pro-cyclic and hence abelian. Therefore the $p$-class field tower of $K$ terminates at $H_K$ and so $H_E=H_K$ for any unramified $p$-extension $E$ of $K$. Hence, if $L\subseteq H_K$, then $(L/K,\emptyset)$ satisfies (C1).\

\noindent{}(ii) Assume $A_K$ is trivial and set $r_K := \mathrm{dim}_{\mathbb{F}_p}(\mathcal{O}_K^\times/(\mathcal{O}_K^\times)^p)$. Then, for any $s\in [r_K+1]$, the Gras-Munnier Theorem (cf. \cite[Prop. 3.1]{GrasMunnier}, \cite{GrasMunnier2}) implies the existence of infinitely many sets $\Sigma$ of  non-archimedean, non $p$-adic, places of $K$ for which $|\Sigma| = s$, $G_{K,\Sigma}$ is a non-trivial cyclic group and $G_{K,\Sigma'}$ is trivial for all 
$\Sigma' \subsetneq \Sigma$. In this case every place in $\Sigma$ is totally ramified in $K_{\Sigma}$ and so, for any intermediate field $L$ of $K_{\Sigma}/K$, one has $A_L = (0)$ so that $(L/K,\emptyset)$ satisfies (C1). \end{examples}

More generally, the following observation leads to many examples in which (C1) is satisfied and $S$ does not contain all places that are either $p$-adic or ramify in $L$.   %is `small'. 

%In that case, we have moreover the following facts.

\begin{lemma}\label{(C1)} Let $L/K$ and $S$ be as in (\ref{data fix}). If there exists a place $\mathfrak{q}$ of $K$ that does not split in $H_{L,S}$, then the following claims are valid.
\begin{itemize}
\item[(i)] $A_{E,S}$ is cyclic (so that $(L/K,S)$ has property (C1)).
\item[(ii)] $A_{E,S}$ is generated by the unique prime $\mathfrak{q}_E$ of $E$ above $\mathfrak{q}$.
\item[(iii)] $G(E/K)$ acts trivially on $A_{E,S}$.
\end{itemize}
\end{lemma} 

\begin{proof}
 If $A_{E,S} = G(H_{E,S}/E)$ is not cyclic, then no place of $E$ can have full decomposition group in $G(H_{E,S}/E)$. In particular, as $H_{E,S}\subseteq H_{L,S}$, this contradicts the existence of $\mathfrak{q}$ and so proves claim (i). Since $\mathfrak{q}_E$ is unramified in $H_{E,S}$, claim (ii) follows directly from class field theory. Claim (iii) then follows from claim (ii) and the fact $\mathfrak{q}_E$ is invariant under the action of $G(E/K)$.
\end{proof}

%The examples of $L/K$ and $S=\emptyset$ satisfying the hypothesis of Lemma \ref{(C1)} have been studied with the theory of pro-$p$ extensions of number fields with restricted ramification.

\begin{remark} If $A_K$ is cyclic, then there are infinitely many sets $\Sigma$ of non-archimedean, non-$p$-adic, places of $K$ for which the (finite) extension $K_{\Sigma}/K$ satisfies the non-splitting hypothesis of Lemma \ref{(C1)}. To see this, write $\mathrm{Gov}(K)$ for the governing field of $K$ (cf. \cite[Def. 3.1]{GrasMunnier}). Then, by applying the Chebotarev Density Theorem to the finite extension $\mathrm{Gov}(K)H_K/K(\zeta_p)$, one can choose a non-archimedean, non $p$-adic, place $\mathfrak{p}$ of $K$ that is inert in $H_K$, splits completely in $K(\zeta_p)$ and is such that, for any, and therefore every, fixed place $\mathfrak{q}$ of $K(\zeta_p)$ above $\mathfrak{p}$, the Frobenius automorphism $\mathrm{Fr}_{\mathfrak{q}}$ of $\mathfrak{q}$ in $ V:= G(\mathrm{Gov}(K)/K(\zeta_p))$ is non-trivial (it is possible that $H_K\subseteq \mathrm{Gov}(K)$, but this fact has no impact on our construction of $\mathfrak{p}$). Now fix $s\in [r_K+1]$. Then, as the $\mathbb{F}_p$-space $V$ has dimension $r_K+1$, one can fix a subset $\{v_i\}_{i \in [s+1]}$ of $V$, with $v_{s+1} = \mathrm{Fr}_{\mathfrak{q}}$, that spans a subspace of dimension $s$ and is such that any proper subset is linearly independent. Let $\Sigma = \{ \mathfrak{p}_i \}_{i \in [s+1]}$ be a set formed by choosing $\mathfrak{p}_{s+1}=\mathfrak{p}$ and a non-$p$-adic place $\mathfrak{p}_i$ of $K$ for each $i \in [s]$ such that $\mathfrak{p}_i$ splits in $K(\zeta_p)$ and $v_i$ is equal to the Frobenius automorphism at a place of $K(\zeta_p)$ above $\mathfrak{p}_i$. Then $G_{K,\Sigma}$ has generator rank $2$ as a consequence of the Gras-Munnier Theorem, the cyclicity of $A_K$, and the cyclicity of the inertia subgroup of the Galois group at a non-$p$-adic place for pro-$p$ extensions of number fields. In addition, by construction, $\mathfrak{p}$ is inert in $H_K$ and ramified in the degree $p$ extension of $K$ (in the Gras-Munnier Theorem) that is ramified precisely at $\Sigma$. Hence, $\mathfrak{p}$ does not split in $K_{\Sigma}$ by the Burnside Basis Theorem.  
%Then, if for $i \in [s-1]$ one fixes a place $\mathfrak{p}_i$ of $K$ with $\mathrm{Fr}_{\mathfrak{p}_i} = v_i$, the Gras-Munnier Theorem and the cyclicity of $A_K$ imply that the generator rank of $G_{K,\Sigma}$ for the set $\Sigma := \{\mathfrak{p}_i\}_{i \in [s-1]}\cup \{\mathfrak{p}\}$ is $2$. By the Burnside Basis Theorem, $\mathfrak{p}$ does not split in $K_{\Sigma}$.
\end{remark}

\begin{remark}\label{Wingberg-exam}
Fix a number field $F$ and a finite set $\Sigma$ of places of $F$ containing all places that are either $p$-adic or archimedean. Then, following Wingberg \cite{Wingberg}, the group $G_{F,\Sigma}$ is said to be `of local type' if some place $\mathfrak{p}$ in $\Sigma$ has full decomposition group in $G_{F,\Sigma}$. In this case, since $\mathfrak{p}$ does not split in $F_{\Sigma}$, Lemma \ref{(C1)} implies that (C1) is satisfied by any cyclic $p$-extension $L/K$ with $F \subseteq K \subseteq L \subset F_{\Sigma}$. In addition, if $F$ is totally real, then \cite[Prop. 1.1]{Wingberg} implies $G_{F,\Sigma}$ is of local type if and only if $F$ is $p$-rational and $\Sigma$ is primitive (in the sense of \cite[\S IV.3]{Gras2}, \cite{Movahhedi}) and so a recent conjecture of Gras \cite{Gras} implies there should be many such $G_{F,\Sigma}$. More generally, \cite{Wingberg} gives a criterion in terms of the arithmetic of $F$ for the group $G_{F,\Sigma}$ to be of local type and explicit examples of such $F$ for which $G_{F,\Sigma}$ is `large' (such as a Demushkin group of rank $4$). 
\end{remark}

%\begin{remark} \textcolor{red}{Not all degree $p$ extensions validate (C1). For example, if $p$ is odd, and $q$ and $\ell$ are primes congruent to $1$ modulo $p$ that are mutually $p$-th power residues, then the (unique) degree $p$ Galois extension $L$ of $\mathbb{Q}$ that is ramified at precisely $q$ and $\ell$ does not validate (C1).}
%\end{remark}

\section{Galois cohomology}\label{subsectionGmodulestructure} 

In this section, we fix data as in (\ref{data fix}) and establish (in Proposition \ref{generators}) the key consequence that Hypothesis \ref{key hyp} has for our theory. To do so, we fix a subgroup $J$ of $G$, set $E:=L^J$ and use the following notations.
\begin{itemize}
\item[-] $I_E$ is the pro-$p$ completion of the group of fractional $\mathcal{O}_E$-ideals.
\item[-] $P_E$ is the pro-$p$ completion of the group of principal $\mathcal{O}_E$-ideals.
\item[-] For a finite set $S$ of primes of $K$, $\langle S \rangle_E$ is the $\mathbb{Z}_p$-submodule of $I_E$ generated by the prime ideals of $E$ above $S$.
\item[-] We write $I_{E,S}$ and $P_{E,S}$ to denote $I_E/\langle S \rangle_E$ and $P_E/(P_E \cap \langle S \rangle_E)$ respectively.
\item[-] By abuse of notations, we use $I_{E,S}$ and $P_{E,S}$ to denote also their images in $I_{L,S}$ under the map $I_{E,S} \to I_{L,S}$ induced by the lifting map.
\item[-] For a fractional ideal $\rho$ of $\mathcal{O}_E$, we will also use $\rho$ to denote its image in $I_E$.
\end{itemize}
We regard all of the groups listed above as $\mathbb{Z}_p[G(E/K)]$-modules in the natural way. 

The following result gives an easy consequence of (C2) regarding these modules that will form the basis of our approach. (We note that all results in this section are vacuously true for the trivial subgroup $J$ and so we will only consider the case that $J$ is non-trivial in the proofs.)

\begin{lemma}\label{kl iso} If $L/K$ satisfies (C2), then there exists a canonical identification of $\mathbb{Z}_p[G/J]$-modules
\[ H^1(J, U_{L,S}) \cong (P_{L,S})^{J}/P_{E,S} \cong \ker\bigl( (I_L)^J/\langle S \rangle_E P_E \xrightarrow{\iota} A_{L,S} \cong I_L/\langle S \rangle_L P_L  \bigr),\] 
where $\iota$ is induced by the natural map $I_{L} \to A_{L,S} \cong I_L/\langle S \rangle_L P_L $. 
\end{lemma}

\begin{proof} There is a canonical exact sequence 
\[ 0 \to \mathcal{O}_{L,S}^\times \to L^\times \to P(L)/P_S(L) \to 0,\]
where for each intermediate field $E$ of $L/K$, $P(E)$ denotes the group of principal fractional $\mathcal{O}_E$-ideals and $P_S(E)$ is the subgroup of principal fractional ideals generated by $S$-units. By Galois cohomology and Hilbert's Theorem 90, we have an exact sequence  
\[\xymatrix{ 0 \ar[r] & \mathcal{O}_{E,S}^\times \ar[r] & E^\times \ar[r] & \big ( P(L)/P_S(L) \big )^J \ar[r] & H^1(J,\mathcal{O}_{L,S}^\times) \ar[r] &0,} \]
and hence an induced isomorphism 
\[ H^1(J, \mathcal{O}_{L,S}^{\times}) \cong \mathrm{coker} \bigl( P(E)/P_S(E) \to (P(L)/P_S(L))^J \bigr).\] %$H^1(J,\mathcal{O}_{L,S}^\times) \cong (P(L)/P_S(L))^J/(P(E)/P_S(E))$.
Upon passing to pro-$p$ completions, this identifies $H^1(J,(\mathcal{O}_{L,S}^{\times})_p)$ with $\mathrm{coker}(P_{E,S} \to (P_{L,S})^J)$.

%$H^1(J, \mathbb{Z}_p \otimes_{\mathbb{Z}} \mathcal{O}_{L,S}^{\times}) \cong (P_{L,S})^{J}/P_{E,S}$. 

We next recall (from Remark \ref{herbrand rem}) that (C2) implies $\mu_L$ is a cohomologically-trivial $G$-module, and hence that the group $H^i(J,\mu_{L,p})$ vanishes for every $i$. From the tautological short exact 
sequence $0 \to \mu_{L,p} \to (\mathcal{O}_{L,S}^\times)_p \to U_{L,S}\to 0$, we can thus deduce that, if (C2) is satisfied, then the natural map $H^1(J,(\mathcal{O}_{L,S}^{\times})_p) \to H^1(J,U_{L,S})$ is bijective. The above argument therefore shows that $H^1(J, U_{L,S})$ is isomorphic to $(P_{L,S})^{J}/P_{E,S}$, as claimed. 

Finally, we note that $(I_{L,S})^J \cap P_{L,S} = (P_{L,S})^J$ and hence that $(P_{L,S})^J/P_{E,S}$ is the kernel of the natural map $(I_{L,S})^J/P_{E,S} \to I_{L,S}/P_{L,S}$. We have $(I_{L,S})^J/P_{E,S} \cong (I_L)^J/\langle S \rangle_E P_E$ because $(I_{L,S})^J$ identifies with $(I_L)^J/\langle S \rangle_E$ since $H^1(J, \langle S \rangle_L)$ vanishes and $S \cap R_{L/K} = \emptyset$. Therefore, the second claimed isomorphism follows.
\end{proof} 

Via this result, the group $H^1(J, U_{L,S})$ is parametrised, under Hypotheses \ref{key hyp}, in terms of the classes in $I_{L}/\langle S \rangle_E P_{E}$ of certain $J$-invariant ideals in $I_{L}$, and in the next result we describe this parametrisation explicitly. 

We assume henceforth that Hypothesis \ref{key hyp} is satisfied and use the following notation. 

\begin{itemize}
\item[-] For $\mathfrak{a}\in I_E$ and $\mathfrak{b}\in (I_{L})^J$, we write 
$[\mathfrak{a}]_E$ and $[ \mathfrak{b} ]'_E$ %{\color{red}this notation clashes with the usual notation for `subgroup generated by', as in $\langle S\rangle_E$ for example which I think causes confusion below - so it would be better to change it in the arguments below, maybe to something like $[\mathfrak{b}]'$??}
for their respective images in $A_{E,S}$ and $(I_{L})^J/\langle S \rangle_E P_{E}$. 
\item[-] We fix a prime $\mathfrak{q}_E$ of $E$ not above $R_{L/K}$ whose class generates $A_{E,S}$.
\item[-] The decomposition and inertia subgroups in $G$ of a place $\mathfrak{r}$ of $K$ are $G(\mathfrak{r})$ and $I(\mathfrak{r})$. 
\item[-] For $\mathfrak{p}\in R_{L/K}$ we fix a $\mathfrak{p}$-adic place 
$\mathfrak{p}_L$ of $L$. We then define $J$-invariant ideals by setting  %define integers $e(\mathfrak{p})$ and $g(\mathfrak{p})$ via $|I(\mathfrak{p})| = p^{e(\mathfrak{p})}$ and $[G:G(\mathfrak{p})]= p^{g(\mathfrak{p})}$. We also set  
\[ \mathfrak{p}_{L/E} := {\prod}_{\sigma\in J/(J\cap G(\mathfrak{p}))}\sigma(\mathfrak{p}_L) .\]
\end{itemize}

In the sequel we also write the action of $\mathbb{Z}_p[G/J]$ on $H^{1}(J,U_{L,S})$ additively and, for  $\mathfrak{p}\in R_{L/K}$, we denote the projection map $\mathbb{Z}_p[G/J] \to \mathbb{Z}_p[G/JG(\mathfrak{p})]$ by $\pi_{J}^\mathfrak{p}$. 

\begin{proposition}\label{generators} If Hypothesis \ref{key hyp} is satisfied, then the following claims are valid. 
\begin{itemize}
\item[(i)] The $\mathbb{Z}_p[G/J]$-module $H^{1}(J,U_{L,S})$ is contained (via isomorphism of Lemma \ref{kl iso}) in the $\mathbb{Z}_p[G/J]$-module $(I_{L})^J/\langle S \rangle_E P_{E}$ which is generated by $\{ [ \mathfrak{q}_E ]'_E \} \cup \{ [ \mathfrak{p}_{L/E} ]'_E \}_{\mathfrak{p} \in R_{L/K}}$. 

\item[(ii)] Fix  $m\in \mathbb{Z}$ and $\{ x(\mathfrak{p})\}_{\mathfrak{p}\in R_{L/K}}\subset \mathbb{Z}_p[G/J]$. 
Then  $[\mathfrak{q}_{E}^{m}\prod_{\mathfrak{p}\in R_{L/K}} (\mathfrak{p}_{L/E})^{x(\mathfrak{p})}]'_E = 0$ 
%$m\langle\mathfrak{q}_{L}^{h_L}\rangle + {\sum}_{\mathfrak{p}\in \Sigma_K^L} x(\mathfrak{p})(\langle\mathfrak{p}_{L/E}'\rangle) = 0$ 
only if for all $\mathfrak{p}\in R_{L/K}$ the element $\pi_J^\mathfrak{p}(x(\mathfrak{p}))$ is divisible by $|J\cap I(\mathfrak{p})|$.
%

%The classes of $p^{\mathrm{min}\{e_{i},m\}}\mathfrak{Q}_{i,m}$ in $H^{1}(G_{m},E_{L_{n}} \otimes_{\mathbb{Z}} \mathbb{Z}%_{p})$ for different $i$'s are linearly dependent over $\mathbb{Z}$.
%
%\item[(ii)] Fix  $m\in \mathbb{Z}$ and $\{ u(\mathfrak{p})\}_{\mathfrak{p}\in \Sigma_K^L}\subset \mathbb{Z}[G]$ such that 
%$m\langle\mathfrak{q}_{L}^{h_L}\rangle + {\sum}_{\mathfrak{p}\in \Sigma_K^L} u(\mathfrak{p})\langle\mathfrak{p}_{L/%E}'\rangle = 0$. Then, for each $\mathfrak{p}$, there exists an integer $r(\mathfrak{p})\in |J\cap I(\mathfrak{p})|
%\mathbb{Z}$ with 
%$u(\mathfrak{p})[\mathfrak{p}_{L/E}'] = r(\mathfrak{p})[\mathfrak{p}_{L/E}']$.
%\item[(iv)] Fix $\Sigma\subseteq \Sigma_K^L$ and $\{m(\mathfrak{p})\}_{\mathfrak{p}\in \Sigma}\subset \mathbb{Z}$ and set 
% $\gamma := {{\sum}}_{\mathfrak{p}\in \Sigma}m(\mathfrak{p})\langle\mathfrak{p}_{L/K}'\rangle$. Assume  
%\begin{itemize}
%\item[(a)] there exist subgroups $H$ and $H'$ of $G$ such that $I(\mathfrak{p}) = H$ and $G(\mathfrak{p}) = H'$ for all $%\mathfrak{p}\in \Sigma$, and 
%\item[(b)] $\{m(\mathfrak{p})\}_{\mathfrak{p}\in \Sigma} \not\subset p\mathbb{Z}$, and 
%\item[(c)] $|J\cap H|\gamma = 0$. 
%\end{itemize}
%Then $\gamma$ generates a $\mathbb{Z}_p[G/J]$-module direct summand of $H^{1}(J, U_L)$  
%isomorphic to $(\mathbb{Z}/|J\cap H|)[G/JH']$.  
\item[(iii)] Fix subgroups $H$ and $H'$ of $G$ and let $\Omega$ be a subset of $R_{L/K}$ with the property that, for all  $\mathfrak{p}$ in $\Omega$, one has $I(\mathfrak{p}) = H$ and $G(\mathfrak{p}) = H'$. Then, if $|\Omega| \ge 3$, there exists a $\mathbb{Z}_p[G/J]$-module direct summand of $H^{1}(J, U_{L,S})$ that is isomorphic to the direct sum of $|\Omega|-2$ copies of $(\mathbb{Z}/|J\cap H|)[G/JH']$.
\end{itemize}
\end{proposition}

\begin{proof} By definition, $(P_{L,S})^J/P_{E,S}$ is a subset of $(I_{L,S})^J/P_{E,S}$. We can analyze $(I_{L,S})^J/P_{E,S}$ by the exact sequence
\begin{equation}\label{phi seq} 0 \to A_{E,S} \to (I_{L,S})^J/P_{E,S} \xrightarrow{\psi} (I_{L,S})^J/I_{E,S}\to 0\end{equation}
and the isomorphism of $\mathbb{Z}_p[G]$-modules
\begin{equation}\label{decomp}(I_{L,S})^J/I_{E,S} \cong (I_L)^J/I_{E} \cong  {\bigoplus}_{\mathfrak{p} \in R_{L/K}}(\mathbb{Z}/|J\cap I(\mathfrak{p})|)[G/JG(\mathfrak{p})] 
\end{equation}
in which each summand is generated by the class of $\mathfrak{p}_{L/E}$. The first isomorphism follows from the condition $S \cap R_{L/K} = \emptyset$. Claim (i) follows because $A_{E,S}$ is generated by $[\mathfrak{q}_E]_E$.

%Since $\langle \mathfrak{q}_L^{|J \cap G_0|} \rangle = \langle \mathfrak{q}_E \mathcal{O}_L \rangle$ generates $A_E \subseteq I_L^J/P_{E,L}$, if we write $X_{L/F}$ for the $\mathbb{Z}_p[G/J]$-submodule of $I_L^J$ generated by $\mathfrak{q}_L$ and the ideals $\mathfrak{p}_{L/E}$ for $\mathfrak{p} \in \Sigma_F^L$, then $H^1(J,U_L) = \{\langle \mathfrak{a}\rangle: \mathfrak{a} \in P_L \cap X_{L/F}\}$. Now, any element  $ \mathfrak{a}$ of $X_{L/F}$ has the form $\mathfrak{a} = \mathfrak{q}_L^m{\prod}_{\mathfrak{p}\in \Sigma_F^L} (\mathfrak{p}'_{L/E})^{x'(\mathfrak{p})}$ with $m \in \mathbb{Z}$ and $x'(\mathfrak{p}) \in \mathbb{Z}_p[G/J]$, and so belongs to $P_L$ if and only if $\mathfrak{q}_L^m \in P_L$. Claim (i) is therefore true since $\mathfrak{q}_L^m \in P_L$ if and only if $m$ is divisible by $h_L$. 

To prove claim (ii) we note that for any $\mathfrak{p} \in R_{L/K}$, any prime $\mathfrak{P}$ of $L$ above $\mathfrak{p}$ and any natural number $n$, one has 
\[ \bigl ( {\prod}_{\sigma\in J/(J\cap G(\mathfrak{p}))} \sigma(\mathfrak{P}) \bigr)^n \in I_E \Longleftrightarrow n \,\,\text{ is divisible by }\,\, |J \cap I(\mathfrak{p})|.\] 
In particular, since the action of $G/J$ on $\mathfrak{p}_{L/E}$ factors through $\pi_J^\mathfrak{p}$, we have $\mathfrak{p}^{x(\mathfrak{p})}_{L/E} \in I_{E}$ if and only if $\pi_{J}^{\mathfrak{p}}(x(\mathfrak{p}))$ is divisible by $|J \cap I(\mathfrak{p})|$.

If $[\mathfrak{q}_{E}^{m}\prod_{\mathfrak{p}\in R_{L/K}} (\mathfrak{p}_{L/E})^{x(\mathfrak{p})}]'_E = 0$, then the ideal $\mathfrak{q}_{E}^{m}{\prod}_{\mathfrak{p}\in R_{L/K}}(\mathfrak{p}_{L/E})^{x(\mathfrak{p})}$ represents an element of $\ker(\psi)$. Claim (ii) is thus a consequence of (\ref{decomp}). 

To prove claim (iii), we set $t := |\Omega|$ and label the places in $\Omega$ as $\{\mathfrak{p}_i\}_{i \in [t]}$. We assume, after relabelling if necessary, that $[\mathfrak{p}_1]_L$ has maximum order (in $A_{L,S}$) amongst the elements $\{[\mathfrak{p}_i]_L\}_{i \in [t]}$. Then, for $j \in [t]\setminus \{1\}$, we fix $m(j) = m(L/K,j) \in \mathbb{Z}$ with $[\mathfrak{p}_{j}\cdot\mathfrak{p}_1^{-m(j)}]_L=0 \in A_{L,S}$ so that %the ideal 
\[ \mathfrak{p}_{1,j,E} := \mathfrak{p}_{j,L/E}\cdot\mathfrak{p}_{1,L/E}^{-m(j)} \in \langle S \rangle_L P_L  \cap (I_L)^J \]
defines an element of $H^1(J,U_{L,S})$. Such a linear relation exists because $A_{L,S}$ is cyclic. Next we note that, for $j\in [t]\setminus\{1\}$, there exists a unique ideal $\mathfrak{p}_{1,j,E}^\ast$ of $E$ with $\mathfrak{p}_{1,j,E}^{|J\cap H|} = \mathfrak{p}_{1,j,E}^\ast\mathcal{O}_L$. After relabelling if necessary, we assume the order of $[\mathfrak{p}_{1,2,E}^\ast]_E$ (in $A_{E,S}$) is maximal amongst the orders of $\{[\mathfrak{p}^\ast_{1,j,E}]_E\}_{j \in [t]\setminus \{1\}}$ and then, for $k\in [t]\setminus \{1,2\}$, we fix 
\begin{equation}\label{n exponent} n(k) = n(L/K,E,k) \in \mathbb{Z}\end{equation}
with $[\mathfrak{p}^\ast_{1,k,E}(\mathfrak{p}_{1,2,E}^\ast)^{-n(k)}]_E=0 \in A_{E,S}$ by using the cyclicity of $A_{E,S}$. Setting for each $3 \leq k \leq t$
\begin{align*} \mathfrak{b}_{k,E} = \mathfrak{b}_{k,L/K,E}:=&\, \mathfrak{p}_{1,k,E} \cdot (\mathfrak{p}_{1,2,E})^{-n(k)}\\
 =&\, (\mathfrak{p}_{1,L/E})^{m(2)n(k)-m(k)}\cdot(\mathfrak{p}_{2,L/E})^{-n(k)}\cdot\mathfrak{p}_{k,L/E}\\
 \in&\, (I_L)^{JH'} \cap \langle S \rangle_L P_L \subseteq (I_L)^{J}\cap \langle S \rangle_L P_L\end{align*}
one therefore has 
\begin{equation}\label{ann} |J\cap H|[ \mathfrak{b}_{k,E} ]'_E = [ \mathfrak{p}^\ast_{1,k,E}(\mathfrak{p}_{1,2,E}^\ast)^{-n(k)}\mathcal{O}_L]'_E = 0.\end{equation}
In particular, since the $\mathfrak{p}_{k,L/E}$-component of the decomposition (\ref{decomp}) is generated over $\mathbb{Z}
_p[G/JH']$ by the image of $\mathfrak{b}_{k,E}$ under the map $\psi$ in (\ref{phi seq}), the $\mathbb{Z}_p[G/J]$-submodule  
 of $H^1(J,U_{L,S})$ generated by $[\mathfrak{b}_{k,E}]'_E$  is isomorphic to $(\mathbb{Z}/|J\cap H|)[G/JH']$. 
 
Next we note that, setting $\Omega' := R_{L/K} \setminus \{\mathfrak{p}_i\}_{i \in [t]\setminus \{1,2\}}$,  claim (i) implies every element of $H^1(J,U_{L,S})$ is represented by an ideal in $\langle S \rangle_L P_L \cap (I_L)^J$ of the form 
\begin{multline}\label{rep ideal} \bigl(\mathfrak{q}_E^m {\prod}_{\mathfrak{p}\in \Omega'}(\mathfrak{p}_{L/E})^{x(\mathfrak{p})}\bigr)\times \bigl({\prod}_{k\in [t]\setminus \{1,2\}}(\mathfrak{b}_{k,E})^{x(k)}\bigr)\\ = \mathfrak{q}_E^m {\prod}_{\mathfrak{p}\in \Omega'}(\mathfrak{p}_{L/E})^{x(\mathfrak{p})'}{\prod}_{k\in [t]\setminus \{1,2\}} (\mathfrak{p}_{k,L/E})^{x(k)},\end{multline}
for elements $x(k)$ of $\mathbb{Z}_p[G/J]$ and suitable integers $m$ and elements $x(\mathfrak{p})$ of $\mathbb{Z}_p[G/J]$. Here, to ensure the equality, we have set  
\[ x(\mathfrak{p})' := \begin{cases} x(\mathfrak{p}) + {\sum}_{k\in [t]\setminus \{1,2\}}(m(2)n(k)-m(k))x(k), &\text{if $\mathfrak{p} = \mathfrak{p}_1$};\\
 x(\mathfrak{p}) - {\sum}_{k\in [t]\setminus \{1,2\}}n(k)x(k), &\text{if $\mathfrak{p} = \mathfrak{p}_2$};\\
 x(\mathfrak{p}), &\text{if $\mathfrak{p} \in \Omega' \setminus\{ \mathfrak{p}_1, \mathfrak{p}_2\}.$}\end{cases}\]

Now, since the ideal (\ref{rep ideal}) represents the trivial class in $A_{L,S}$, the first factor in the product on the left hand side must belong to the group  
\[ I^\ast := \langle S \rangle_L P_L\cap \{\mathfrak{q}_E^m {\prod}_{\mathfrak{p}\in \Omega'}(\mathfrak{p}_{L/E})^{x(\mathfrak{p})}: m \in \mathbb{Z},\,\,\,x(\mathfrak{p})\in \mathbb{Z}_p[G/J]\} \subseteq (I_L)^J.\]
We now write $X$ and $Y$ for the $\mathbb{Z}_p[G/J]$-submodules of $H^1(J,U_{L,S})$ that are respectively generated by the classes of ideals in $I^\ast$ and $\{\mathfrak{b}_{k,E}\}_{k \in [t]\setminus \{1,2 \}}$. Then, to prove the claim, it is enough to show that $H^1(J,U_{L,S})$ decomposes as a direct sum of $\mathbb{Z}_p[G/J]$-modules $X\oplus Y$ and that $Y$ is isomorphic to a direct sum of $t-2$ copies of $(\mathbb{Z}/|J\cap H|)[G/JH']$. 

To show this, it is in turn enough to assume the ideal (\ref{rep ideal}) has trivial class in $H^1(J,U_{L,S})$, and thereby deduce every element $\pi_{J}^\mathfrak{p}(x(k))$ is divisible by $|J\cap H|$ and the ideal $\bigl(\mathfrak{q}_E^m \prod_{\mathfrak{p}\in \Omega'}\mathfrak{p}_{L/E}^{x(\mathfrak{p})} \bigr)$ has trivial class in $H^1(J,U_{L,S})$. The first condition follows directly upon applying claim (ii) to the right hand side of (\ref{rep ideal}). Since this combines with (\ref{ann}) to imply $[ \prod_{k \in [t]\setminus\{1,2\}}(\mathfrak{b}_{k,E})^{x(k)}]'_E = 0$, the vanishing of $[ \mathfrak{q}_E^m \prod_{\mathfrak{p}\in \Omega'}(\mathfrak{p}_{L/E})^{x(\mathfrak{p})}]'_E$ then follows from the left hand expression in (\ref{rep ideal}).  
\end{proof} 

\section{Module structures}

In this section, we fix a natural number $n$ and a cyclic group $\Gamma$ of order $p^n$, with generator $\sigma$. For $i \in [n]^\ast$, we write $\Gamma_i$ for the subgroup of $\Gamma$ generated by $\sigma^{p^{n-i}}$ (so that $|\Gamma_i| = p^i$). 

We write ${\rm Lat}_n$ for the category of $\mathbb{Z}_p[\Gamma]$-lattices and fix a set of representatives $\mathcal{I}_n$ of the isomorphism classes of indecomposable $\mathbb{Z}_p[\Gamma]$-lattices that contains $\mathbb{Z}_p[\Gamma/\Gamma_i]$ for every $i \in [n]^\ast$.

\subsection{Yakovlev diagrams}\label{sec-Yak}

We write $\mathfrak{M}_n$ for the category of diagrams 
\begin{equation*}\label{yak diags} (A_\bullet,\alpha_\bullet,\beta_\bullet): \,\, A_1 \stackrel[\beta_1]{\alpha_1}{\leftrightarrows} A_2 \stackrel[\beta_2]{\alpha_2}{\leftrightarrows}  \cdots \stackrel[\beta_{n-1}]{\alpha_{n-1}}{\leftrightarrows} A_n\end{equation*}
in which each $A_i$ is a finite $(\mathbb{Z}/p^i)[\Gamma/\Gamma_i]$-module, and each $\alpha_i$ and $\beta_i$ is a morphism of $\mathbb{Z}_p[\Gamma]$-modules such that $\beta_i\circ \alpha_i$ and $\alpha_i\circ \beta_i$ are respectively induced by multiplication by $p$ and by the action of ${\sum}_{\gamma\in \Gamma_{i+1}/\Gamma_{i}}\gamma$. A morphism $(A_\bullet,\alpha_\bullet,\beta_\bullet) \to (A'_\bullet,\alpha'_\bullet,\beta'_\bullet)$ in $\mathfrak{M}_n$ is a collection of maps $\{\gamma_i:A_i\to A_i'\}_{i\in [n]}$ of $\mathbb{Z}_p[\Gamma]$-modules that commute with the respective maps 
 $\alpha_\bullet, \beta_\bullet, \alpha_\bullet',\beta_\bullet'$ (in particular, such a morphism is an isomorphism if and only if each map $\gamma_i$ is bijective). 

As a concrete example, each $M$ in ${\rm Lat}_n$ gives an object 
\[\Delta(M) = (A_\bullet,\alpha_\bullet,\beta_\bullet)\]
of $\mathfrak{M}_n$ in which each $A_i$ is $H^1(\Gamma_i,M)$ and each $\alpha_i$ and $\beta_i$ the natural  
restriction and corestriction maps. The importance of such examples is explained by the following result.  

\begin{proposition}[{Yakovlev \cite{Yakovlev1}}]\label{yak prop} The assignment $M \mapsto \Delta(M)$ induces a covariant essentially surjective functor $\Delta: {\rm Lat}_n \to \mathfrak{M}_n$. In addition, if  $\Delta(M)$ and $\Delta(N)$ are isomorphic, then there are non-negative integers $\{m_i\}_{i \in [n]^\ast}$ and $\{m_i'\}_{i \in [n]^\ast}$ and an isomorphism in ${\rm Lat}_n$ of the form   
\[ M \oplus {\bigoplus}_{i \in [n]^\ast} \mathbb{Z}_p[\Gamma/\Gamma_i]^{m_i} \cong N \oplus {\bigoplus}_{i \in [n]^\ast} \mathbb{Z}_p[\Gamma/\Gamma_i]^{m'_i}.\] 
\end{proposition} 

\begin{remark}\label{krull cons} When combined with the Krull-Schmidt Theorem (for the category ${\rm Lat}_n$), the final assertion of this result implies that 
if $\Delta(M)$ and $\Delta(N)$ are isomorphic, then any lattice in $\mathcal{I}_n$ that occurs (with a given multiplicity) as a direct summand of $M$ is either $\mathbb{Z}_p[\Gamma/\Gamma_i]$ for some $i$ or occurs (with the same multiplicity) as a direct summand of $N$. In particular, the isomorphism class of $\Delta(U_{L,S})$ in $\mathfrak{M}_n$ determines, uniquely up to isomorphism, a module $U_{L,S}^\dagger$ in ${\rm Lat}_n$ that has no direct summand isomorphic to $\mathbb{Z}_p[\Gamma/\Gamma_i]$ for any $i\in [n]^\ast$ and is such that for some (uniquely determined) set $\{t_i\}_{i \in [n]^\ast}$ of non-negative integers, there exists an isomorphism in ${\rm Lat}_n$ of the form 
\begin{equation}\label{ks iso} U_{L,S} \cong U_{L,S}^\dagger \oplus {\bigoplus}_{i \in [n]^\ast}\mathbb{Z}_p[\Gamma/\Gamma_i]^{t_i}.\end{equation}
\end{remark}

The next result presents an explicit example that will be useful in the next section. 

\begin{lemma}\label{explicit yak diagrams} Fix $a\in [n]$ and a non-negative integer $b$ with $a+b\le n$ and set $c:= n-(a+b)$. 
\begin{itemize}
\item[(i)] If $b=0$, then $M_{a,0} := \mathbb{Z}_p[\Gamma](\sigma^{p^{c}}-1)$ is an indecomposable $\mathbb{Z}_p[\Gamma]$-lattice.
\item[(ii)] If $b \not=0$, then  $M_{a,b} := \mathbb{Z}_p[\Gamma](p^{a}, \sigma^{p^{c}}-1)$ is an indecomposable $\mathbb{Z}_p[\Gamma]$-lattice. 
\item[(iii)] For all $a$ and $b$, the morphisms 
\[ H^1(\Gamma_i,M_{a,b}) \xrightarrow{{\rm res}} H^1(\Gamma_{i-1},M_{a,b}) \xrightarrow{{\rm cor}} H^1(\Gamma_i,M_{a,b})\]
are equivalent to 
\[\begin{cases} (\mathbb{Z}/p^{i})[\Gamma/\Gamma_{a+b}] \xrightarrow{} (\mathbb{Z}/p^{i-1})[\Gamma/\Gamma_{a+b}] \xrightarrow{\times p} (\mathbb{Z}/p^{i})[\Gamma/\Gamma_{a+b}], &\text{if $1< i\le a$,}\\
(\mathbb{Z}/p^a)[\Gamma/\Gamma_{a+b}] \xrightarrow{{\rm id}} (\mathbb{Z}/p^a)[\Gamma/\Gamma_{a+b}] \xrightarrow{\times p} (\mathbb{Z}/p^a)[\Gamma/\Gamma_{a+b}], &\text{if $a< i\le a+b$,}\\
(\mathbb{Z}/p^a)[\Gamma/\Gamma_{i}] \xrightarrow{T_i} (\mathbb{Z}/p^a)[\Gamma/\Gamma_{i-1}] \xrightarrow{} (\mathbb{Z}/p^a)[\Gamma/\Gamma_{i}], &\text{if $a+b< i\le n.$}
\end{cases}\]
Here the two unlabelled arrows are the natural projection maps and $T_i$ sends each element $\gamma$ of $\Gamma/\Gamma_i$ to the sum in $(\mathbb{Z}/p^a)[\Gamma/\Gamma_{i-1}]$ of all elements of 
$\Gamma/\Gamma_{i-1}$ that project to $\gamma$.   
\end{itemize}
\end{lemma} 
 
\begin{proof} We set $I_{a,b} := \mathbb{Z}_p[\Gamma](p^{a}, \sigma^{p^{c}}-1)$ (so that $I_{a,b} = M_{a,b}$ if $b\not= 0$). Then, in all cases, there is an 
exact sequence of $\mathbb{Z}_p[\Gamma]$-modules $0 \longrightarrow I_{a,b} \longrightarrow \mathbb{Z}_{p}[\Gamma] \longrightarrow (\mathbb{Z}/p^{a})[\Gamma/\Gamma_{a+b}] \longrightarrow 0$. Upon taking  $\Gamma_i$-cohomology of this sequence, one obtains an exact sequence of $\mathbb{Z}_p[\Gamma/\Gamma_i]$-modules 
\begin{equation*}
(\mathbb{Z}_{p}[\Gamma])^{\Gamma_{i}} \longrightarrow \big ( (\mathbb{Z}/p^{a})[\Gamma/\Gamma_{a+b}] \big )^{\Gamma_{i}} \longrightarrow H^{1}(\Gamma_{i},I_{a,b}) \longrightarrow 0.
\end{equation*}
A direct calculation using these sequences shows that the morphisms 
\[ H^1(\Gamma_i,I_{a,b}) \xrightarrow{{\rm res}} H^1(\Gamma_{i-1},I_{a,b}) \xrightarrow{{\rm cor}} H^1(\Gamma_i,I_{a,b})\]
are equivalent to the morphisms in claim (iii). Since all of the $\mathbb{Z}_p[\Gamma]$-modules that occur in this description have cyclic $\Gamma$-coinvariants,  Nakayama's lemma (for the local ring $\mathbb{Z}_p[\Gamma]$) implies that they are each indecomposable. Hence, as all occurring maps are non-zero, the diagram $\Delta(I_{a,b})$ must itself be indecomposable. In particular, if $I_{a,b}$ is decomposable, then it must have the form $I_{a,b} = N_1\oplus N_2$ for an indecomposable module $N_1$ with $\Delta(N_1)\cong \Delta(I_{a,b})$, and $H^1(\Gamma_i,N_2) = 0$ for all $i \in [n]^\ast$. Remark \ref{krull cons} then implies that $N_2$ is isomorphic to a direct sum $\bigoplus_{t \in [n]^\ast}\mathbb{Z}_p[\Gamma/\Gamma_t]^{m_t}$ for suitable integers $m_t$. Now, if $N_2\not= (0)$, then, as $\mathbb{Q}_p\otimes_{\mathbb{Z}_p}I_{a,b} = \mathbb{Q}_p[\Gamma]$, there exists a unique $s\in [n]^\ast$ for which $m_s = 1$ and $m_t = 0$ for all $t\in [n]^\ast \setminus \{s\}$ and so $I_{a,b} \cong N_1\oplus \mathbb{Z}_p[\Gamma/\Gamma_s]$. This decomposition implies $\Sigma_s := {\sum}_{\gamma \in \Gamma_s}\gamma \in \mathbb{Z}_p[\Gamma]$ acts as the zero map on $N_1$ and hence that $I_{a,b}$ is preserved by the action of $p^{-s}\Sigma_s$. By explicit check, one finds this can only happen if $a=s$ and $b=0$ so that $I_{a,b} = M_{a,0} \oplus \mathbb{Z}_p[\Gamma] ({\sum}_{\gamma \in \Gamma_a} \gamma)$. In addition, $M_{a,0}$
 is a cyclic module over the local ring $\mathbb{Z}_p[\Gamma]$ and so indecomposable. This verifies all of the stated claims. \end{proof}

\subsection{The proof of Theorem \ref{intro thm}} \label{unit group section}

We write $\mathcal{C}_n$ for the class of pairs $(L/K, S)$ comprising a cyclic extension $L/K$ of number fields of degree $p^n$ and a finite set $S$ of places of $K$ satisfying Hypothesis \ref{key hyp} and, for every $(L/K,S)$ in $\mathcal{C}_n$, we fix an identification of $G:= G(L/K)$ with $\Gamma$. %By the condition (C2), we can study $U_{L,S}$ with the Yakovlev diagram. 
 Before starting the proof of Theorem \ref{intro thm}, we recall that Heller and Reiner \cite{HellerReiner2} have shown  $\mathcal{I}_n$, the set of indecomposable $\mathbb{Z}_p[\Gamma]$-lattices fixed before, to be infinite if 
$n > 2$ and note that, as $(L/K,S)$ ranges over $\mathcal{C}_n$, the rank $\mathrm{rk}(U_{L,S})$ is unbounded. Given these facts, the result of Theorem \ref{intro thm} is therefore, a priori, far from clear.  

Turning now to its proof, we introduce some useful notation. Firstly, $J$ and $J'$ will henceforth always denote non-trivial subgroups of $\Gamma$ (as the group $H^1(\{1\}, U_{L,S})$ vanishes and so plays no role in Yakovlev's theory); then, for subgroups $H \subseteq H'$ of $\Gamma$, and each $(L/K,S)$ in $\mathcal{C}_n$, we set 
\[ \Omega_{L/K}^{H,H'} := \{\mathfrak{p} \in R_{L/K}: I(\mathfrak{p}) = H, G(\mathfrak{p}) = H'\}\quad\text{and}\quad t_{L/K}^{H,H'} := |\Omega_{L/K}^{H,H'}|.\] 
Then, assuming $t_{L/K}^{H,H'} > 2$, the argument of 
Proposition \ref{generators}(iii) gives, for each subgroup $J$ of $\Gamma$, and $k \in t_{L/K}^{H,H'}\setminus \{1,2\}$, an ideal 
\[ \mathfrak{b}^{H,H'}_k(L/K,S,J)\in (I_L)^{JH'} \cap \langle S \rangle_L P_L  \subseteq (I_L)^J\cap \langle S \rangle_L P_L \] 
as follows: with $\mathfrak{b}^{H,H'}_k(L/K,S,H')$ denoting the ideal $\mathfrak{b}_{k,L^{H'}}$ constructed in the proof of Proposition \ref{generators}(iii) for the case $\Omega = \Omega^{H,H'}_{L/K}$, we set 
\begin{equation} \mathfrak{b}^{H,H'}_k(L/K,S,J) = \begin{cases} \mathfrak{b}^{H,H'}_k(L/K,S,H') &\text{if } J \subseteq H', \\ 
 {\rm Norm}_{L^{H'}/E}(\mathfrak{b}^{H,H'}_k(L/K,S,H')) &\text{if } H' \subsetneq J.\end{cases}\end{equation}
%
%If  $J \subseteq H'$, then $E(\mathfrak{p}) = L^{H'}$ and we set $\mathfrak{b}_k(L/K,J) := \mathfrak{b}_{k,L/K,L^J}$.
 %If $H'\subsetneq J$, then $E(\mathfrak{p}) = L^J$ and $\mathfrak{p}^\ast_{1,j,E} = {\rm Norm}_{L^{H'}/E}(\mathfrak{p}^\ast_{1,j,L^{H'}})$ for $j \in t_{L/K}^{H,H'}\setminus \{1\}$ so that we may fix the integers $n(L/K,L^J,k)$ in (\ref{n exponent}) to be independent of 
 %$J$ and then set $\mathfrak{b}_k(L/K,J) := \mathfrak{b}_{k,L/K,L^J}.$ 
% 
With these choices, for subgroups $J' < J$ with $|J/J'| = p$, one has  
\begin{equation}\label{relations} {\prod}_{\sigma \in J/J'}\sigma(\mathfrak{b}^{H,H'}_k(L/K,S,J')) = \begin{cases} \mathfrak{b}^{H,H'}_k(L/K,S,J)^p &\text{if } J \subseteq H', \\ 
 \mathfrak{b}^{H,H'}_k(L/K,S,J) &\text{if } H' \subsetneq J.\end{cases}\end{equation}
For each $J \le \Gamma$, $H \subseteq H'$ with $|t_{L/K}^{H,H'}|> 2$ and $k \in [t_{L/K}^{H,H'}]\setminus \{1,2\}$, we set 
\[ B^{H,H'}_k(L/K,S,J) := \mathbb{Z}_p[\Gamma/J] [ \mathfrak{b}^{H,H'}_k(L/K,S,J)]'_E \subseteq H^1(J,U_{L,S}).\]
Then, by the argument of Proposition \ref{generators}(iii), the module $B^{H,H'}_k(L/K,S,J)$ is isomorphic to $(\mathbb{Z}/|J\cap H|)[\Gamma/JH']$ and there exists an isomorphism of $\mathbb{Z}_p[\Gamma/J]$-modules
\begin{equation}\label{key decomp} H^1(J,U_{L,S}) \cong C_J \oplus {\bigoplus}_{(H,H')\in \Upsilon_{L/K}} {\bigoplus}_{k \in [t_{L/K}^{H,H'}]\setminus \{1,2\}}B^{H,H'}_k(L/K,S,J).  \end{equation} 
Here $\Upsilon_{L/K}$ denotes the collection of pairs $(H,H')$ of subgroups of $\Gamma$ with $H \le H'$ and $t_{L/K}^{H,H'}>2$ and $C_J$ is generated by the classes arising from (suitable) products of the conjugates of ideals in the set 
\[\Omega_{L/K,S,J}^\ast := \{ \mathfrak{q}_E\}\cup \{ \mathfrak{p}_{L/E} : \mathfrak{p} \in R_{L/K}, (I(\mathfrak{p}),G(\mathfrak{p}))\notin \Upsilon_{L/K}\} \cup \{ \mathfrak{p}_{L/E} : \mathfrak{p} \in {\bigcup}_{(H,H')}\Omega_{H,H'} \} \]
where $(H,H')$ runs over $\Upsilon_{L/K}$ and $\Omega_{H,H'}$ is a subset of $\Omega_{L/K}^{H,H'}$ of cardinality $2$ (corresponding to the first two places in the ordering of $\Omega_{L/K}^{H,H'}$ fixed in the proof of Proposition \ref{generators}(iii)).  

Next we note that, with respect to the identification in Lemma \ref{kl iso}, for subgroups $J' \le J$, the corestriction $H^1(J',U_{L,S}) \to H^1(J,U_{L,S})$ and restriction $H^1(J,U_{L,S})\to H^1(J',U_{L,S})$ maps are respectively induced by the map $(I_{L})^{J'} \to (I_L)^J$ sending each $\mathfrak{a}$ to $\prod_{\sigma\in J/J'}\sigma(\mathfrak{a})$ and by the inclusion 
$(I_L)^{J} \subseteq (I_L)^{J'}$. From (\ref{relations}), we can therefore deduce that, as $J$ varies, the decompositions (\ref{key decomp}) respect the relevant restriction and corestriction maps and so  induce a decomposition of Yakovlev diagrams 
\begin{equation*}\label{almost} \Delta(U_{L,S}) = \Delta_1(U_{L,S}) \oplus {\bigoplus}_{(H,H')\in \Upsilon_{L/K}} {\bigoplus}_{k \in [t_{L/K}^{H,H'}]\setminus \{1,2\}} \Delta_{k}^{H,H'}(U_{L,S}),\end{equation*}
where $\Delta_1(U_{L,S})$ is constructed from $\{C_J\}_{J \le \Gamma}$ and $\Delta_{k}^{H,H'}(U_{L,S})$ from $\{B_k^{H,H'}(L/K,S,J)\}_{J \le \Gamma}$. In addition, by comparing 
(\ref{relations}) to the result of Lemma \ref{explicit yak diagrams}(iii), one deduces that each diagram $\Delta_k^{H,H'}(U_{L,S})$ is isomorphic to $\Delta (M_{H,H'})$, with $M_{H,H'} := M_{a,b}$ for integers $a$ and $b$ specified by $|H| = p^{a}$ and $|H'/H| = p^{b}$. The above decomposition is thus equivalent to an isomorphism 
\begin{equation}\label{yak iso} \Delta(U_{L,S}) \cong \Delta_1(U_{L,S}) \oplus \Delta\bigl({\bigoplus}_{(H,H')\in \Upsilon_{L/K}} 
(M_{H,H'})^{ (t_{L/K}^{H,H'} -2)}\bigr)\end{equation}
in $\mathfrak{M}_n$. Finally we note that $|\Omega_{L/K,S,J}^\ast| \le 1 + 2B_n$ with $B_n$ the number of subgroup pairs $(H,H')$ of $\Gamma$ and hence that the order of each (finite) group $C_J$ is bounded by a number that depends only on $p$ and $n$. Thus, as $(L/K,S)$ ranges over $\mathcal{C}_n$, the number of isomorphism classes of possible diagrams $\Delta_1(U_{L,S})$ is also bounded by a number depending only on $p$ and $n$. Hence, by combining the isomorphism (\ref{yak iso}) with the observation in Remark \ref{krull cons}, we can finally deduce that, as $(L/K,S)$ ranges over $\mathcal{C}_n$, the number of modules in $\mathcal{I}_n$ that can arise in the Krull-Schmidt decomposition of at least one of the lattices $U_{L,S}$ is finite. This completes the proof of Theorem \ref{intro thm}.

%\begin{proposition}\label{counting}
%The number of possible Yakovlev diagrams of $E_{L_{n}} \otimes_{\mathbb{Z}} \mathbb{Z}_{p}$ is bounded above by the product of the number of possible equivalence class of the diagrams
%\begin{equation*}
%\begin{tikzcd}
%Y_{m} : 1},E_{L_{n}} \otimes_{{H^{1}(G_{m}, E_{L_{n}} \otimes_{\mathbb{Z}} \mathbb{Z}_{p})} \arrow[rrrr, "\mathrm{cor}", %shift left] &  &  &  & {H^{1}(G_{m+\mathbb{Z}} \mathbb{Z}_{p})} \arrow[llll, "\mathrm{res}", shift left]
%\end{tikzcd}
%\end{equation*}
%over $1 \leq m \leq n-1$.
%\end{proposition}

%\section{Galois structure of units in cyclic $p$-extensions of number fields}

%A closer analysis in case (I) gives the following result.  

%\begin{proposition} If $L/K$ is an unramified extension of $p$-rtr fields, then the $\mathbb{Z}_p[G]$-module $U_L$ is %isomorphic to $(\mathbb{Z}_p[G])^{\oplus (d-1)} \oplus \mathbb{Z}_p[G]/({\sum}_{g \in G}g),$ with $d = [K:\mathbb{Q}]$.
%\end{proposition}
%\begin{proof} 
%\end{proof}  

\subsection{Some special cases}

There are at least two situations in which a closer analysis of the above argument can give more information. Firstly, if $n$ is `small', then the $\mathbb{Z}[\Gamma]$-module structures of terms in $\Delta(U_{L,S})$ are relatively simple and the categories ${\rm Lat}_n$ and $\mathfrak{M}_n$ are even completely understood for $n \in \{1,2\}$ (cf. \cite{hr} and \cite[Th. 5]{Yakovlev1} respectively). Hence, the argument of Proposition \ref{generators}(iii) can sometimes give an effective means of obtaining the full Krull-Schmidt decomposition of $U_{L,S}$  (see \cite{KumonLim} for results in this direction for cyclic extensions of degree dividing $p^3$). Secondly, if $H_{L,S}$ is a cyclic extension of $K$, then the argument in \S\ref{unit group section} can be simplified and leads to the following result.

%We remark that for subfields $F \subseteq K \subset K' \subseteq L$, if $A_{K',S} \neq 0$, then the primes of $S$ split completely in $H_{K',S} \subset H_{L,S}$. Hence, we have $H_{K',S}=H_{K,S}$ and the transfer map $A_{K,S} \to A_{K',S}$ is surjective. 

%$A_{L,S} \neq 0$, then the primes of $S$ split completely in $L$ and $L/F$ is unramified. Hence, we have $H_{L,S}=H_{F,S}$ and the tr

%Under this condition, for subfields $F \subseteq K \subset K' \subseteq L$ with $[K':K]=p$, either $h_{K'}=h_K=1$ or $ph_{K'}=h_K$.

\begin{theorem}\label{last thm} Let $L/K$ be a finite cyclic $p$-extension of number fields with ${\rm Norm}_{L/K}(\mu_L) = \mu_K$ and $S$ a finite set of places of $K$  such that $H_{L,S}/K$ is cyclic. Then the $\mathbb{Z}_p[G(L/K)]$-module structure of $U_{L,S}$ depends only on the ramification and inertia degrees of the places in $S \cup R_{L/K}$. In particular, if $L/K$ is unramified and all places in $S$ split completely in $L$, then there exists an isomorphism of  $\mathbb{Z}_p[G(L/K)]$-modules 
\[ U_{L,S}\cong \bigl(\mathbb{Z}_p[G(L/K)]/({\sum}_{g \in G(L/K)}g)\bigr)\oplus \mathbb{Z}_p[G(L/K)]^{({\rm rk}(U_K) + |S|)}.\]
\end{theorem}

\begin{proof} We assume $[L:K] = p^n$ and fix an identification $G(L/K) = \Gamma$. We also write $G(S)$ for the subgroup  $\prod_{\tau \in S} G(\tau)$ of $\Gamma$ that is generated by $\bigcup_{\tau \in S}G(\tau)$. 

At the outset we note that, if $A_{L,S} \neq (0)$, then, as $H_{L,S}/K$ is a cyclic $p$-extension, $L/K$ is unramified (so $R_{L/K} =\emptyset$), all places in $S$ split completely in $L$ and, for every intermediate field $E$ of $L/K$, one has $H_{E,S}=H_{L,S}$. On the other hand, if $A_{L,S} = (0)$, then $H_{E,S}$ is the maximum unramified extension of $E$ in $L$ in which all places in $S$ split completely. 

Next we note that, since $(L/K,S)$ satisfies the hypothesis of Lemma \ref{(C1)}, the results of Lemma \ref{(C1)}(ii) and (iii) imply that $G(E/K)$ acts trivially on $A_{E,S}$ and that the (equivalence class of) transfer and norm maps between the respective groups $\{A_{E,S}\}_E$ are uniquely determined by the orders of each group $A_{E,S}$. In particular, in this way one finds that each transfer map $A_{E,S} \to A_{L,S}$ is surjective and that $$|\ker(A_{E,S} \to A_{L,S})| = |(J G(S){\prod}_{\mathfrak{p} \in R_{L/K}} I(\mathfrak{p}))/(G(S){\prod}_{\mathfrak{p} \in R_{L/K}} I(\mathfrak{p}))|.$$

In the remainder of the argument, we consider separately the cases $R_{L/K}\not=\emptyset$ and $R_{L/K}=\emptyset$. Thus, until further notice, we assume  $R_{L/K}\not=\emptyset$. We write $L'$ for the maximal unramified extension of $K$ in $L$ and set $\Gamma' := G(L/L')$ and $Z := U_{L,S}$. In this case one has $A_{L,S}=(0)$ and so the module $H^1(J,Z) = (P_{L,S})^J/P_{E,S} = (I_{L,S})^J/P_{E,S}$ lies in a canonical short exact sequence
\begin{equation}\label{cap ses}0 \to A_{E,S} \to H^1(J,Z) \to (I_{L,S})^J/I_{E,S} \to 0\end{equation}
and $(I_{L,S})^J/I_{E,S} \cong (I_L)^J/I_E$ is explicitly known via the isomorphism (\ref{decomp}). 

For $\mathfrak{p}\in R_{L/K}$ we set $J_0(\mathfrak{p}) := J \cap I(\mathfrak{p})$ and $J_1(\mathfrak{p}) := J \cap G(\mathfrak{p})$. Then one has  $\mathfrak{p}_{L/E}^{|J_0(\mathfrak{p})|} = \mathfrak{p}_E\mathcal{O}_L$, with $\mathfrak{p}_E = \mathfrak{p}_L \cap \mathcal{O}_E$ and so the element $|J_0(\mathfrak{p})|[ \mathfrak{p}_{L/E}]'_E$ of $H^1(J,Z)$ is represented (via (\ref{cap ses})) by the class $[\mathfrak{p}_E]_E\in A_{E,S} \subseteq H^1(J,Z)$. In addition, all of these classes $\{[\mathfrak{p}_E]_E\}$ are related to the single class $[\mathfrak{p}_{L(\mathfrak{p})}]_{L(\mathfrak{p})}$ with  $L(\mathfrak{p}) = L^{G(\mathfrak{p})}$ by norm and transfer maps. If $J\subseteq I(\mathfrak{p})$, then $A_{E,S} = (0)$, if $I(\mathfrak{p})\subseteq J \subseteq G(\mathfrak{p})$, then $\mathfrak{p}_E = \mathfrak{p}_{L(\mathfrak{p})}\mathcal{O}_E$ and if $G(\mathfrak{p})\subseteq J$, then $\mathfrak{p}_E = {\rm Norm}_{L(\mathfrak{p})/E}(\mathfrak{p}_{L(\mathfrak{p})})$.  
For every $E$, the index in $A_{E,S}$ of the subgroup generated by $[\mathfrak{p}_E]_E$ is equal to 
\[ e(J, \mathfrak{p}) : = (J\Gamma' G(S) : J_1(\mathfrak{p}) \Gamma'G(S)).\] 
Let us fix a place $\mathfrak{q} \not\in R_{L/K} \cup S$ of $K$ that is inert in $H_{L,S}$ and write $\mathfrak{q}_E$ for the place of $E$ above $\mathfrak{q}$. Then, for each $\mathfrak{p} \in R_{L/K}$, Lemma \ref{(C1)}(ii) implies there exists an integer $u(\mathfrak{p})$ that is prime to $p$ and such that  
\[ [(\mathfrak{p}_{L/L(\mathfrak{p})})^{|I(\mathfrak{p})|u(\mathfrak{p})} ]'_{L(\mathfrak{p})} = [ \mathfrak{q}^{e(G(\mathfrak{p}), \mathfrak{p})}_{L(\mathfrak{p})} ]'_{L(\mathfrak{p})} =  [  \mathfrak{q}_{L(\mathfrak{p})} ]'_{L(\mathfrak{p})}.\] 
In particular, by the preceding remark, for every $E$ one has 
\[ [ (\mathfrak{p}_{L/E})^{|J_0(\mathfrak{p})|u(\mathfrak{p})} ]'_E = [ \mathfrak{q}^{e(J, \mathfrak{p})}_{E} ]'_E.\] 

Now the argument of the proof of Proposition \ref{generators} (i) and (ii) implies $H^{1}(J,Z)$ is isomorphic to the $\mathbb{Z}[\Gamma/J]$-module $W_{J}$ with generators $\{ Y_J\}\cup \{ X_{\mathfrak{p},J}\}_{\mathfrak{p} \in R_{L/K}}$ and relations 
\[ |A_{E,S}|Y_{J}, \,\, \sigma Y_{J} = Y_{J},\,\, |J_0(\mathfrak{p})|X_{\mathfrak{p},J}= e(J, \mathfrak{p})Y_{J},\,\,\sigma^{|\Gamma/(G(\mathfrak{p})J)|}X_{\mathfrak{p},J}=X_{\mathfrak{p},J}.\]
Furthermore, by the construction of $\{u(\mathfrak{p})\}_{\mathfrak{p} \in R_{L/K}}$ and $\{\mathfrak{q}_E\}_{J < \Gamma}$, these presentations of $H^1(J,Z)$ are compatible with varying $J$ in the following sense. If $J'$ is the subgroup of $J$ of index $p$, then the restriction $H^{1}(J, Z) \to H^{1}(J', Z)$ and corestriction $H^{1}(J', Z) \to H^{1}(J, Z)$ maps correspond to the homomorphisms $\alpha_{J}: W_{J} \to W_{J'}$ and $\beta_J: W_{J'} \to W_{J}$ specified by 
\begin{align*}
\alpha_{J}(Y_{J}) = Y_{J'} \quad \text{and} \quad \alpha_{J}(X_{\mathfrak{p},J}) = &\begin{cases} X_{\mathfrak{p},J'} \, 
&\text{if } J \subseteq G(\mathfrak{p}) \\ T_{J/J'}(X_{\mathfrak{p},J'}) \, &\text{if } G(\mathfrak{p}) \subsetneq J,\end{cases}\\
\beta_{J}(Y_{J'}) = pY_{J} \quad \text{and} \quad \beta_{J}(X_{\mathfrak{p},J'}) = &\begin{cases} pX_{\mathfrak{p},J} \, 
&\text{if }  J \subseteq G(\mathfrak{p}) \\ X_{\mathfrak{p},J} &\text{if }  G(\mathfrak{p}) \subsetneq J. \end{cases} 
\end{align*}
This analysis shows that the isomorphism class in $\mathfrak{M}_n$ of the diagram $\Delta(U_{L,S})$ depends only on the groups $I(\mathfrak{p})$ and $G(\mathfrak{p})$ for $\mathfrak{p}$ in $R_{L/K}$, and $G(S)$. Hence, recalling the decomposition (\ref{ks iso}), to determine the isomorphism class of $Z$ itself, it is sufficient to determine the integers $t_i$. For $j \in [n]^{\ast}$, let $s_j$ be the number of primes $\mathfrak{p}$ of $S$ such that $\Gamma_j$ is the decomposition subgroup in  $\Gamma$ of the place of $L$ above $\mathfrak{p}$. Then one can determine the integers $t_i$ by using the fact, for each $j \in [n]^{\ast}$, that
\begin{align}\label{t_i formula} {\rm rk}\bigl(Z^{\Gamma_j}) - {\rm rk}\bigl(Z^{\Gamma_{j+1}})
=&\, {\rm rk}((Z^\dagger)^{\Gamma_j}) - {\rm rk}((Z^\dagger)^{\Gamma_{j+1}}) + (p^{n-j} - p^{n-j-1}) {\sum}_{i \in [j]^\ast} \, t_i \\
=&\, (p^{n-j} - p^{n-j-1}) ({\rm rk}(U_K) + 1 + {{\sum}}_{i \in [j]^\ast} \, s_i) \notag\end{align}
%\begin{align}\label{t_i formula} &\, {\rm rk}\bigl((U_{L,S})^{\Gamma_j}) - {\rm rk}\bigl((U_{L,S})^{\Gamma_{j+1}})\\
%=&\, {\rm rk}((U_{L,S}^\dagger)^{\Gamma_j}) - {\rm rk}((U_{L,S}^\dagger)^{\Gamma_{j+1}}) + (p^{n-j} - p^{n-j-1}) {\sum}_{i \in [j]^\ast} \, t_i \notag\\
%=&\, (p^{n-j} - p^{n-j-1}) ({\rm rk}(U_K) + 1 + {{\sum}}_{i \in [j]^\ast} \, s_i) \notag\end{align}
where the last equality follows by applying the Dirichlet-Herbrand Theorem to the $\mathbb{Q}_p[\Gamma]$-module $\mathbb{Q}_p\otimes_{\mathbb{Z}_p} U_{L,S}$ (cf. \cite[Th. I.3.7]{Gras2}). Since the isomorphism class of $Z^{\dagger}$ is determined by $\Delta(U_{L,S})$, we can determine $\{t_i\}_{i \in [n]^{\ast}}$ recursively.

%\begin{align*}
%p^{n-j} \bigl( {\rm rk}_{\mathbb{Z}_p}U_K + 1 \bigr) + {{\sum}}_{\mathfrak{p} \in S} p^n/{\rm max}\{ p^j, |G(\mathfrak{p})|\} \\
%=&\, {\rm rk}_{\mathbb{Z}_p}((U_{L,S}^\dagger)^{\Gamma_j}) + {{\sum}}_{i \in [n]^\ast} t_i\cdot {\rm rk}_{\mathbb{Z}_p}\bigl( (\mathbb{Z}_p[\Gamma/\Gamma_i])^{\Gamma_j}\bigr)\\ 
%                                            =&\, {\rm rk}_{\mathbb{Z}_p}((U_{L,S}^\dagger)^{\Gamma_j}) + p^{n-j}\bigl({{\sum}}_{0 \le i \le j} t_i\bigr)  + {{\sum}}_{j < i \le n}t_i p^{n-i}.\end{align*}

%
%\begin{align*} {\rm rk}_{\mathbb{Z}_p}\bigl((U_{L,S})^{\Gamma_j}) =&\, p^{n-j} \bigl( {\rm rk}_{\mathbb{Z}_p}U_K + 1 \bigr) + {{\sum}}_{\mathfrak{p} \in S} p^n/{\rm max}\{ p^j, |G(\mathfrak{p})|\} \\
%=&\, {\rm rk}_{\mathbb{Z}_p}((U_{L,S}^\dagger)^{\Gamma_j}) + {{\sum}}_{i \in [n]^\ast} t_i\cdot {\rm rk}_{\mathbb{Z}_p}\bigl( (\mathbb{Z}_p[\Gamma/\Gamma_i])^{\Gamma_j}\bigr)\\ 
%                                            =&\, {\rm rk}_{\mathbb{Z}_p}((U_{L,S}^\dagger)^{\Gamma_j}) + p^{n-j}\bigl({{\sum}}_{0 \le i \le j} t_i\bigr)  + {{\sum}}_{j < i \le n}t_i p^{n-i}.\end{align*}

This establishes the claimed result in the case $R_{L/K} \not= \emptyset$ and so in the rest of the argument we assume $R_{L/K} = \emptyset$ (so that $L/K$ is unramified). In this case, for a subgroup $J$ of $\Gamma$ one has $(I_{L,S})^J/I_{E,S}=(0)$ and so $H^1(J,Z)$ identifies with the kernel of the transfer map $A_{E,S} \to A_{L,S}$. By the observation made at the beginning of this proof, for each $i \in [n]^\ast$ the morphisms 
\[ H^1(\Gamma_i,Z) \xrightarrow{{\rm res}} H^1(\Gamma_{i-1},Z) \xrightarrow{{\rm cor}} H^1(\Gamma_i,Z)\]
are uniquely determined by $G(S)$. In particular, if all places of $S$ split completely in $L$, then one has $H_{E,S}=H_{L,S}$ for all intermediate fields $E$ of $L/K$ and the above morphisms identify with $\mathbb{Z}/p^i \xrightarrow{} \mathbb{Z}/p^{i-1} \to \mathbb{Z}/p^{i}$, where the first arrow is the natural projection map and the second sends $1$ to $p$. Given this description, an easy exercise shows that $\Delta(Z)$ is isomorphic in $\mathfrak{M}_n$ to $\Delta(N)$ for the indecomposable $\mathbb{Z}_p[\Gamma]$-lattice $N := \mathbb{Z}_p[\Gamma]/({\sum}_{\gamma \in \Gamma}\gamma)$. Just as above, it then follows that, for suitable (uniquely determined) integers $\{t_i\}_{i \in [n]^\ast}$, there is an 
isomorphism in ${\rm Lat}_n$
\[ Z \cong N \oplus {\bigoplus}_{i \in [n]^\ast} \mathbb{Z}_p[\Gamma/\Gamma_i]^{ t_i}.\] 
In addition, the Dirichlet-Herbrand Theorem for $S$-units implies the existence of an isomorphism 
\[ \mathbb{Q}_p \otimes_{\mathbb{Z}_p} Z \cong (\mathbb{Q}_p \otimes_{\mathbb{Z}_p} N) \oplus \mathbb{Q}_p[\Gamma]^{({\rm rk}(U_K) + |S|)}\]
 of $\mathbb{Q}_p[\Gamma]$-modules. Upon comparing these isomorphisms, and noting $\mathbb{Q}_p[\Gamma]$ is semisimple, it follows that $t_i = 0$ for $i \not= 0$ and $t_0 = {\rm rk}(U_K) + |S|$. This implies the claimed isomorphism.
\end{proof}

\section{Minkowski units}\label{Minkowski section} 

%Let $L/K$ be a Galois extension of number fields and $S$ a finite set of primes of $K$. Then, if the $\mathbb{F}_p[G(L/K)]$-module $U_{L,S}/U_{L,S}^p$ has a direct summand that is isomorphic to $\mathbb{F}_p[G(L/K)]^{m}$ for some natural number $m$, one says that $L/K$ has a family of `$m$ independent Minkowski $S$-units'. In particular, the Krull-Schmidt Theorem implies that the maximum number $m_{L/K,S}$ of linearly independent Minkowski units for $L/K$ is a well-defined non-negative integer.  
%
%In recent work of Hajir, Ramakrishna and the third author, it is shown that, as a consequence of the Gras-Munnier Theorem and Kummer theory, knowledge of $m_{L/K,S}$ plays  an important role in the study of tamely ramified pro-$p$ extensions of $L$, and in particular in the study of the deficiency of $p$-class tower groups (see \cite{HMR1}) and inverse Galois problem for the $p$-class field tower (cf. \cite{HMR2, Ozaki}). In this direction, our approach has the concrete consequence described in the next result. 

In this final section, we derive a consequence of Theorem \ref{intro thm} regarding the existence of independent Minkowski units (as discussed in the Introduction). 

To state the result, we use the family of field extensions $\mathcal{C}_n$ defined at the beginning of \S\ref{unit group section}. For $(L/K,S)$ in $\mathcal{C}_n$ we recall, from the argument in \S\ref{unit group section}, that $\Upsilon_{L/K}$ denotes the set of subgroup pairs of $\Gamma$ that arise as $(I(\mathfrak{p}),G(\mathfrak{p}))$ for at least {\em three} distinct places in $R_{L/K}$ and we set  
\[ R_{L/K}^{(3)} := \{ \mathfrak{p} \in R_{L/K}: (I(\mathfrak{p}),G(\mathfrak{p}))\in \Upsilon_{L/K}\}.\]
We also write $r_1(K)$ and $r_2(K)$ for the respective numbers of real and complex places of $K$, and $n_{S,L}$ for the number of places in $S$ that split completely in $L$.

\begin{corollary}\label{cft remark} There exists a natural number $N_{p,n}$ that depends only on $p$ and $n$ and has the following property: for each $(L/K,S)$ in  $\mathcal{C}_n$, one has   
\[ m_{L/K,S} = r_1(K)+r_2(K) + n_{S,L} + (2|\Upsilon_{L/K}| - |R^{(3)}_{L/K}|) + d_{L/K,S}\]
with $|d_{L/K,S}|\le N_{p,n}$.
\end{corollary}

\begin{proof}  The isomorphism (\ref{yak iso}) in $\mathfrak{M}_n$ implies the existence of a module $M_{L/K} = M_{L/K,S}$ in ${\rm Lat}_n$ for which there is an isomorphism (in ${\rm Lat}_n$) of the form 
\begin{equation}\label{iso-dagger}
U_{L,S}^{\dagger} \cong M_{L/K} \oplus {\bigoplus}_{(H,H') \in \Upsilon_{L/K}}(M_{H,H'})^{(t_{L/K}^{H,H'} -2)},
\end{equation}
and one has $\mathrm{rk}(M_{L/K}) \le N'_{p,n}$ for an integer $N'_{p,n}$ that depends only on $p$ and $n$. 

We first claim that, for each of the $\mathbb{Z}_p[\Gamma]$-lattices $M_{a,b}$ in Lemma \ref{explicit yak diagrams}, the corresponding $\mathbb{F}_p[\Gamma]$-module $M_{a,b}/pM_{a,b}$  does not contain $\mathbb{F}_p[\Gamma]$ as a direct summand. To see this, note ${\rm rk}(M_{a,b}) \leq p^n$ and so $M_{a,b}/pM_{a,b}$ can have a direct summand isomorphic to $\mathbb{F}_p[\Gamma]$ only if $M_{a,b} \cong \mathbb{Z}_p[\Gamma]$ and, since $M_{a,b}$ is not cohomologically-trivial, this is not true.  

Hence, with $t_0$ the integer that occurs in $(\ref{ks iso})$, the isomorphism (\ref{iso-dagger}) implies that the non-negative integer $m_{L/K,S}-t_0$ is bounded above by the multiplicity of $\mathbb{F}_p[\Gamma]$ in 
$M_{L/K}/pM_{L/K}$. In particular, since $\mathrm{rk}(M_{L/K})$ is bounded solely in terms of $p$ and $n$, there exists a natural number $N''_{p,n}$ that depends only on $p$ and $n$ and is such that $0 \le m_{L/K,S}-t_0 \le N''_{p,n}.$ 
%Since , and so there exists a natural number which can again be bounded solely in terms of $p$ and $n$.

Next we set $r_1 = r_1(K)$ and $r_2 = r_2(K)$ and note that the formula (\ref{t_i formula}) (in which the term $s_0$ is equal to $n_{S,L}$) implies that 
\begin{equation}\label{first rank equality}
 (p^n-p^{n-1})(r_1+r_2+ n_{S,L}) =\, {\rm rk}(U_{L,S}^{\dagger}) - {\rm rk}\big ( (U_{L,S}^{\dagger})^{\Gamma_1}\big ) + (p^n-p^{n-1})t_0.\end{equation} 
%
%Now the isomorphism (\ref{yak iso}) in $\mathfrak{M}_n$ implies the existence of an isomorphism in ${\rm Lat}_n$ 
%
%\begin{equation*}\label{dagger iso} U_{L,S}^{\dagger} \cong M_{L/K,S} \oplus {\bigoplus}_{(H,H') \in \Upsilon_{L/K}}(M_{H,H'})^{(t_{L/K}^{H,H'} -2)}\end{equation*}
%
%for some lattice $M_{L/K,S}$ for which $\mathrm{rk}(M_{L/K,S})$ is bounded independently of $(L/K,S)$.

In addition, for $(H,H') \in \Upsilon_{L/K}$, a straightforward computation (using the fact $\mathbb{Q}_p\otimes_{\mathbb{Z}_p}I_{a,b} = \mathbb{Q}_p[\Gamma]$ for each of the lattices $I_{a,b}$ that occur in the proof of Lemma \ref{explicit yak diagrams}) shows that
\[ {\rm rk}(M_{H,H'}) - {\rm rk}((M_{H,H'})^{\Gamma_1})=p^n-p^{n-1}.\] 

From the isomorphism $(\ref{iso-dagger})$, one therefore deduces that    
\[ {\rm rk}(U_{L,S}^{\dagger}) - {\rm rk}\big ( (U_{L,S}^{\dagger})^{\Gamma_1}\big ) =  d'_{L/K,S} + {\sum}_{(H,H') \in \Upsilon_{L/K}}\!\!\!\! (p^n-p^{n-1})( t_{L/K}^{H,H'} - 2 ),\]
with $d'_{L/K,S} := {\rm rk}(M_{L/K}) - {\rm rk}(M_{L/K}^{\Gamma_1})$ (so that $0 \le d'_{L/K,S} \le N'_{p,n}$). Upon substituting this into (\ref{first rank equality}), and dividing the resulting equality by $p^n-p^{n-1}$, one deduces that   
\begin{align*} t_0 =&\, r_1+r_2 + n_{S,L} - {\sum}_{(H,H') \in \Upsilon_{L/K}}\!\!\!\! ( t_{L/K}^{H,H'} - 2 ) - d'_{L/K,S}/(p^n-p^{n-1})\\
                       =&\, r_1+r_2 + n_{S,L} + (2|\Upsilon_{L/K}| - |R^{(3)}_{L/K}|) - d'_{L/K,S}/(p^n-p^{n-1}),\end{align*}
where the second equality is true since ${\sum}_{(H,H') \in \Upsilon_{L/K}}t_{L/K}^{H,H'} = |R^{(3)}_{L/K}|,$ 
as follows from a direct comparison of the definitions of the terms $R_{L/K}^{(3)}$, $\Upsilon_{L/K}$ and $t_{L/K}^{H,H'}$. 

The claimed result is therefore obtained by setting 
\[ d_{L/K,S} := -d'_{L/K,S}/(p^n-p^{n-1}) + (m_{L/K,S} - t_0)\]
and taking $N_{p,n}$ to be the integer part of $N'_{p,n}/(p^n-p^{n-1}) + N''_{p,n}$. 
\end{proof}

In \cite[\S 5]{HMR3}, the authors construct families of Galois extensions in which Minkowski units can be shown to exist. We now finish this section by using Corollary \ref{cft remark} to describe new families of extensions in which there are many Minkowski units. In particular, the following examples show that, for each $n$, the quantity $m_{L/K,S}$ is unbounded as $(L/K,S)$ ranges over $\mathcal{C}_n$. We remark that these examples are qualitatively different from those in \cite{HMR3} since the existence of Minkowski units is not being forced either by tame ramification or by large numbers of ramified places. 

\begin{examples}\label{last exams} In order to show that $m_{L/K,S}$ is unbounded as $(L/K,S)$ ranges over $\mathcal{C}_n$ it is sufficient, by Corollary \ref{cft remark}, to identify families of $(L/K,S)$ for which $|R_{L/K}|$ is bounded but $r_1(K)+r_2(K)+ n_{S,L}$ is unbounded. In particular, for a fixed extension $L/K$, the quantity $n_{S,L}$, and hence also $m_{L/K,S}$, is clearly unbounded as one increases the set $S$. Of more interest, however, is the fact (evidenced by the following examples) that the required conditions are also satisfied in cases with $S=\emptyset$. 

\begin{itemize}
%\item[(i)]  Since the set $R_{L/K}$ is fixed, $m_{L/K,S}$ tends to infinity as $|S| \to \infty$.
\item[(i)] Assume $F$ has a unique $p$-adic place $\mathfrak{p}$ and $A_F = (0)$. Then, for the set $\Sigma = \{ \mathfrak{p}\}$, $G_{F,\Sigma}$ is the inertia subgroup of the unique $p$-adic place of $F_{\Sigma}$ (cf. Example \ref{Wingberg-exam}). Hence, $\mathfrak{p}$ is totally ramified in $F_{\Sigma}$ and $A_L=(0)$ for every finite extension $L$ of $F$ in $F_{\Sigma}$. The quantity $m_{L/K,\emptyset}$ is therefore unbounded as $L/K$ ranges over intermediate fields of the tower $F_{\Sigma}/F$ since each such extension is ramified precisely at the unique $p$-adic place. In addition, for each such $L/K$ and a subgroup $J$ of $G(L/K)$, $(P_L)^J/P_E \cong (I_L)^J/I_E$ is generated by the class of the prime of $L$ above $\mathfrak{p}$. By the argument in \S\ref{unit group section}, $\Delta(U_{L}) \cong \Delta(M_{n,0})$ for the lattice $M_{n,0}$ in Lemma \ref{explicit yak diagrams}. Hence, there is an isomorphism of $\mathbb{Z}_p[G(L/K)]$-modules  
\[ U_{L}\cong \bigl(\mathbb{Z}_p[G(L/K)]/({\sum}_{g \in G(L/K)}g)\bigr)\oplus \mathbb{Z}_p[G(L/K)]^{{\rm rk}(U_{K}) }.\]
\item[(ii)] Let $p$ and $q$ be distinct primes, with $p$ odd and $q \equiv 1 \pmod{p}$. Then, if both $q \not \equiv 1 \bmod{p^2}$ and $p$ is not a $p$-th power modulo $q$, the Burnside Basis Theorem implies that neither $p$ nor $q$ can split in the pro-$p$ extension $\mathbb{Q}_{\{p, q\}}/\mathbb{Q}$. In this case, therefore, Corollary \ref{cft remark} implies that the quantity $m_{L/K,\emptyset}$ is unbounded in the family of cyclic extensions $L/K$ with $L \subset \mathbb{Q}_{\{p, q\}}$ since each such extension is ramified at at most two places.   
\end{itemize}
\end{examples}

\begin{remark}\label{end remark} The mutual congruence conditions on $p$ and $q$ in Examples \ref{last exams}(iii) also arise in the theory of central extensions of number fields (cf. \cite[Th. 5.2]{Frohlich}). The following observation, which we have not been able to find in the literature, is thus perhaps also of interest beyond ensuring the existence of independent Minkowski units. \end{remark}

\begin{proposition}
For each odd prime $p$, there are infinitely many primes $q \equiv 1 \pmod{p}$ such that both $q \not \equiv 1 \pmod{p^2}$ and $p$ is not a $p$-th power modulo $q$.
\end{proposition}

\begin{proof}
Since $\mathbb{Q}(\sqrt[p]{p}, \zeta_p)$ is a non-abelian Galois extension of $\mathbb{Q}$, the fields $\mathbb{Q}(\sqrt[p]{p}, \zeta_p)$ and $\mathbb{Q}(\zeta_{p^2})$ are linearly disjoint over $\mathbb{Q}(\zeta_p)$ and so the group $G(\mathbb{Q}(\zeta_{p^2}, \sqrt[p]{p})/\mathbb{Q}(\zeta_p))$ is isomorphic to $(\mathbb{Z}/p)^2$. 

Let now $q$ be a rational prime such that the Frobenius automorphism $\mathrm{Fr}_{q}$ in $G(\mathbb{Q}(\zeta_{p^2}, \sqrt[p]{p})/\mathbb{Q})$ at a place of $\mathbb{Q}(\zeta_{p^2}, \sqrt[p]{p})$ above $q$ is contained in $G(\mathbb{Q}(\zeta_{p^2}, \sqrt[p]{p})/\mathbb{Q}(\zeta_p))$ and restricts to non-trivial elements in $G(\mathbb{Q}(\zeta_{p^2})/\mathbb{Q}(\zeta_p))$ and $G(\mathbb{Q}(\zeta_p, \sqrt[p]{p})/\mathbb{Q}(\zeta_p))$. Then one has $q \equiv 1 \pmod{p}$ since $\mathrm{Fr}_{q}$ belongs to $G(\mathbb{Q}(\zeta_{p^2}, \sqrt[p]{p})/\mathbb{Q}(\zeta_p))$, and also $q \not \equiv 1 \pmod{p^2}$ because $\mathrm{Fr}_{q}$ acts non-trivially on $\mathbb{Q}(\zeta_{p^2})$. On the other hand, since the restriction of $\mathrm{Fr}_{q}$ in $G(\mathbb{Q}(\zeta_p, \sqrt[p]{p})/\mathbb{Q}(\zeta_p))$ is non-trivial, the Gras-Munnier Theorem implies (via \cite[Th. V. 2.4.2]{Gras2}) that there exists no cyclic extension of $\mathbb{Q}$ of degree $p$ that is ramified precisely at $q$ and in which  $p$ splits. It follows that $p$ does not split in the degree $p$ subfield of $\mathbb{Q}(\zeta_q)$ and so $p$ cannot be a $p$-th power modulo $q$, as required.
%
%Let $\mathfrak{q}$ be a prime of $\mathbb{Q}(\zeta_p)$ whose Frobenius automorphism in $G(\mathbb{Q}(\zeta_{p^2}, \sqrt[p]{p})/\mathbb{Q}(\zeta_p))$ restricts to non-trivial elements in $G(\mathbb{Q}(\zeta_{p^2})/\mathbb{Q}(\zeta_p))$ and $G(\mathbb{Q}(\zeta_p, \sqrt[p]{p})/\mathbb{Q}(\zeta_p))$. Since the set of primes of $\mathbb{Q}(\zeta_p)$ whose inertia degrees over $\mathbb{Q}$ are larger than 1 has Dirichlet density zero, we may assume that the decomposition subgroup of $G(\mathbb{Q}(\zeta_p)/\mathbb{Q})$ at $\mathfrak{q}$ is trivial. Let $q$ be the rational prime below $\mathfrak{q}$. Then $q \equiv 1 \pmod{p}$ by the preceding assumption. Since $\mathfrak{q}$ does not split in $\mathbb{Q}(\zeta_{p^2})$, we have $q \not \equiv 1 \pmod{p^2}$. On the other hand, as the Frobenius automorphism of $G(\mathbb{Q}(\zeta_{p}, \sqrt[p]{p})/\mathbb{Q}(\zeta_p))$ at $\mathfrak{q}$ is non-trivial, there is no cyclic extension of $\mathbb{Q}$ of degree $p$ that is ramified precisely at $q$ and in which  $p$ splits (this follows from a version of the Gras-Munnier Theorem \cite[Th. V. 2.4.2]{Gras2}). Hence, $p$ does not split in the degree $p$ cyclic extension of $\mathbb{Q}$ in $\mathbb{Q}(\zeta_q)$ and so $p$ is not a $p$-th power modulo $q$, as required.
\end{proof}

%This approach can be effective especially when $n$ is small because there are not many different kinds of $(e_{i},f_{i},g_{i})$, and since $e_{i} \geq 1$, a $3$-tuple $(e_{i},f_{i},g_{i})$ with $f_{i} > 0$ can have small $g_{i}$. Therefore, the analysis of the Galois module structure can be relatively easy.

\end{document}